\documentclass[12pt,twoside,reqno]{amsart}
\usepackage{graphicx}

\usepackage{indentfirst}

\usepackage[letterpaper, portrait, margin=1in]{geometry}

\usepackage{amsmath}
\usepackage{amssymb}
\usepackage{amsfonts}
\usepackage{fancyhdr}
\usepackage{MnSymbol}
\usepackage{mathrsfs}
\usepackage{indentfirst}
\usepackage{mathtools} 
\usepackage{commath}
\usepackage{soul}
\usepackage{upgreek} 
\usepackage{graphicx} 
\usepackage[rightcaption]{sidecap}
\usepackage{wrapfig}
\usepackage{caption}
\usepackage{theoremref}
\usepackage{cancel}
\usepackage{amsthm}
\usepackage[final]{hyperref}
\usepackage{float}
\usepackage{comment}
\usepackage[dvipsnames]{xcolor}

\newtheorem{theorem}{Theorem}[section]

\newtheorem{lemma}[theorem]{Lemma}
\newtheorem{definition}[theorem]{Definition}
\newtheorem{example}[theorem]{Example}
\newtheorem{corollary}[theorem]{Corollary}
\newtheorem{remark}[theorem]{Remark}

\def\C{{\mathbb{C}}}
\def\R{{\mathbb{R}}}

\hypersetup{unicode= false, colorlinks=true, linkcolor=NavyBlue,
	anchorcolor=blue, citecolor=ForestGreen, filecolor=red, menucolor=blue, urlcolor=NavyBlue}

\begin{document}
\title{Direct spectral problems for Paley-Wiener canonical systems}
\author{Ashley R. Zhang}
\address{Vanderbilt University\\ Department of Mathematics\\
		1326 Stevenson Center\\
		Nashville, TN  37240\\ USA }
\email{ashley.zhang@vanderbilt.edu}

\begin{abstract}

This note focuses on the direct spectral problem for canonical Hamiltonian systems on the half-line $\mathbb{R}_+$. Truncated Toeplitz operators have been effectively used to solve the inverse spectral problem when the spectral measure is a locally finite periodic measure (see \cite{MP}). Here, we reverse the inverse problem algorithm to solve the direct spectral problem for step-function Hamiltonians. For a non-step-function Hamiltonian, we consider its step-function approximations and their corresponding spectral measures, and show that these spectral measures converge to the spectral measure of the original Hamiltonian.

\end{abstract}
\maketitle

\section{Introduction}

The main problem we study in this note is spectral problems for canonical Hamiltonian systems on the half-line $\mathbb{R}_+$. A standard tool in this field is Krein-de Branges (KdB) theory, first established in the mid-20th century by M. G. Krein. Krein observed numerous connections between structural problems in certain spaces of entire functions and spectral problems for canonical systems. L. de Branges later developed the complex analytic aspects of this theory. The fundamentals of KdB theory and additional references can be found in the original book by de Branges \cite{dB}, in a chapter of the book by Dym and McKean \cite{DM}, and in recent monographs by C. Remling \cite{Remling} and R. Romanov \cite{Rom}.

Over the past two decades, KdB theory has experienced another surge of activity, driven by its applications to various other areas of analysis such as the theory of orthogonal polynomials and the nonlinear Fourier transform (see, for instance \cite{Denisov, Scatter, Zhang}), as well as its connections with various other areas of mathematics including number theory, random matrices, and the zeros of the zeta function (\cite{Burnol, Lagarias1, Lagarias2, Benedek}).

Canonical systems on the half-line $\mathbb{R}_+$ are $2 \times 2$ differential equation systems of the form \begin{equation} \label{CS}
    \Omega \Dot{X}(t) = z H(t) X(t), \ \ t \in (0, t_+),\ \  t_+ \leqslant \infty,
\end{equation}
where $z\in\C$ is a spectral parameter,
\begin{equation*}
    \Omega = \begin{pmatrix} 0 & 1 \\ -1 & 0 \end{pmatrix}
    \end{equation*}
is the symplectic matrix,
\begin{equation*}
   H(t) = \begin{pmatrix} h_{11}(t) & h_{12}(t) \\ h_{21}(t) & h_{22}(t) \end{pmatrix}
\end{equation*}
is a given matrix-valued function called the Hamiltonian of the system, and
\begin{equation*}
    X(t) = \begin{pmatrix} u(t) \\ v(t) \end{pmatrix} 
    \end{equation*}
is the unknown vector function.

We make the following assumptions on the Hamiltonian $H(t)$:
\begin{itemize}
    \item $H(t) \in \mathbb{R}^{2 \times 2}$;
    \item $H(t) \in L^1_{loc}(\mathbb{R})$;
    \item $H(t) \geqslant 0$ almost everywhere;
    \item $\det(H) \neq 0$ almost everywhere.
\end{itemize} 

The first three are standard assumptions on a Hamiltonian, and the fourth is a restriction. This restriction implies that the system does not have any 'jump intervals' (sub-intervals of $(0,t_+)$ on which $H$ is a constant matrix of rank $1$).

With the change of variable $ds=\det H(t) dt$, we can normalize a canonical system to satisfy $\det \left( H(s) \right) = 1$ almost everywhere. We call such systems det-normalized and assume this normalization throughout the rest of the paper. Details on this change of variable can be found for instance in \cite{MP}.

A new approach to the inverse spectral problem for canonical systems was recently introduced in \cite{BR, Bessonov, MP}. This method is based on the study of truncated Toeplitz operators with symbols equal to the spectral measures of canonical systems. Compared to the systems studied in the classical works of Borg \cite{Borg1, Borg2}, Marchenko \cite{Mar1, Mar2}, and Gelfand-Levitan \cite{GL}, this approach broadens the scope of systems that can be analyzed and provides new explicit examples of solutions to the inverse spectral problem. The Toeplitz approach applies to systems whose spectral measures belong to the Paley-Wiener class (PW-class), the class of sampling measures in Paley-Wiener spaces. This is significantly broader than the Gelfand-Levitan class considered in the classical theory.

As discussed in \cite{MP}, this new approach establishes a direct connection between spectral problems and the theory of orthogonal polynomials when the spectral measure is a locally finite periodic measure. One of our goals in this note is to use this connection to solve the direct spectral problem and develop a finite-dimensional algorithm \eqref{Algorithm} that complements the inverse problem algorithm presented in \cite{MP}.

In the case of a locally finite periodic spectral measure, as shown in \cite{MP}, the Hamiltonian of the canonical system is always a step function with a uniform step size. For the inverse spectral problem, one can periodize a non-periodic spectral measure $\mu$ and use the Hamiltonians of these periodizations to approximate the original Hamiltonian. This approach was shown to be effective for measures in the PW-class, and certain measures outside the PW-class, see \cite{PZ}. Our other goal in this note is to use a similar approach to solve the direct spectral problem within the PW-class. For a non-step-function PW-Hamiltonian $H$, we can consider its step-function approximations $H^T$ and apply tools from the step-function case to find spectral measures $\mu_T$. For a particular class of PW-Hamiltonians, those arising from real Dirac systems, we establish $\mu_T \to \mu$ as $T \to 0$. Whether this periodization approach extends to certain Hamiltonians outside the PW-class, similar to several results in \cite{PZ}, remains an interesting open question.

This paper is organized as follows: In section \ref{CSdB}, we review the basics of Krein-de Branges theory of canonical systems, introduce Paley-Wiener systems, and revisit the algorithm for the inverse spectral problem for even periodic PW-measures. In sections \ref{alpha} and \ref{Moments}, we present two parallel algorithms for the direct spectral problem when the Hamiltonian is diagonal and is a step function with a uniform step size. In section \ref{Approx}, we extend the algorithm to non-step-function Hamiltonians. Finally, in section \ref{examples}, we provide examples illustrating our methods and formulas.

\section{Canonical systems and de Branges spaces}\label{CSdB}
\subsection{De Branges spaces}
A solution to \eqref{CS} is a $C^2-$function 
\begin{equation*}
    X(t) = X(t, z) = \begin{pmatrix} u(t, z) \\v(t, z) \end{pmatrix}
\end{equation*}
on $(0, t_+)$ satisfying the equation. An initial value problem (IVP) for \eqref{CS} is given by an initial condition $X(0, z) = c$, $c \in \mathbb{R}^2$. It is well known that every IVP for \eqref{CS} has a unique solution $X(t, z)$ on $(0, t_+)$, see for instance chapter 1 of \cite{Remling}.

We use $H^2(\mathbb{C}_+)$ to denote the Hardy space on the upper half-plane $\mathbb{C}_+$: \begin{equation*}
    H^2(\mathbb{C}_+) = \{ f \in \mathcal{H}(\mathbb{C_+})~|~ \sup_{y > 0} \int_{\mathbb{R}} \abs{f(x + iy)}^2 < \infty\},
\end{equation*}
where $\mathcal{H}(C_+)$ is the set of all analytic functions in $\mathbb{C_+}$.

An entire function $E(z)$ is called a \textit{Hermite-Biehler function} if $|E(z)| > |E(\overline{z})|$ for $z \in \mathbb{C}_+$.
Given an entire function $E(z)$ of Hermite-Biehler class, define the \textit{de Branges space} $B(E)$ based on $E$ as \begin{equation*}
    B(E) = \left\{F \text{~entire}\ |\ F/E, F^\#/E \in H^2(\mathbb{C_+}) \right\},
\end{equation*}
where $F^\#(z) = \overline{F(\overline{z})}$.  The space $B(E)$ becomes a Hilbert space when endowed with the scalar product inherited from $H^2(\C_+)$: \begin{equation*}
    [F, G] = \int_{-\infty}^\infty F(t) \overline{G(t)} \frac{dt}{|E(t)^2|}.
\end{equation*}
We call an entire function real if it is real-valued on $\R$. Associated with $E$ are two real entire functions  $A = \frac{1}{2}(E + E^\#)$ and $C = \frac{i}{2}(E - E^\#)$ such that $E=A-iC$.
De Branges spaces are also reproducing kernel Hilbert spaces. The reproducing kernels 
\begin{equation*}
    K_\lambda(z) = \frac{\overline{E(\lambda)} E(z) - \overline{E^\#(\lambda)} E^\#(z)}{2\pi i(\overline{\lambda} - z)} = \frac{\overline{C(\lambda)} A(z) - \overline{A(\lambda)}C(z)}{\pi(\overline{\lambda} - z)}
\end{equation*}
are functions from $B(E)$ and satisfy $[F, K_{\lambda}] = F(\lambda)$ for all $F \in B(E)$, $\lambda \in \mathbb{C}$.

De Branges spaces have an alternative axiomatic definition, which is useful in many applications.

\begin{theorem}[\cite{dB}] \label{dBAxioms}
Suppose that $H$ is a Hilbert space of entire functions that satisfies \begin{itemize}
\item $F \in H$, $F(\lambda) = 0$ $\Rightarrow$ $F(z) \frac{z - \overline{\lambda}}{z - \lambda} \in H$ with the same norm,
\item $\forall \lambda \notin \mathbb{R}$, the point evaluation is a bounded linear functional on $H$,
\item $F \to F^\#$ is an isometry.
\end{itemize}
Then $H = B(E)$ for some entire function $E$ of Hermite-Biehler class.
\end{theorem}

The most elementary example of de Branges spaces is the Paley-Wiener spaces. We denote by $PW_a$ the standard Paley-Wiener space: the subspace of $L^2(\R)$ consisting of entire functions of exponential type at most $a$. $PW_a = B(E)$ with $E(z) = e^{-iaz}$.
The reproducing kernels of $PW_a$ are the sinc functions,
$$\frac{\sin(t(z-\lambda))}{\pi( z-\lambda)}.$$

\subsection{Spectral measures}
Let $X(t, z) = \begin{pmatrix} u(t, z) \\v(t, z) \end{pmatrix}$  be the unique solution of \eqref{CS} on $(0, t_+)$
satisfying a self-adjoint initial condition $X(0, z) = c$, $c \in \mathbb{R}^2$. Krein first observed that the function $E(t, z) = u(t, z) - i v(t, z)$ is a Hermite-Biehler entire function for every fixed $t$. 

Every det-normalized canonical system delivers a chain of nested de Branges spaces $B(E_t)$: if  $0 < t_1 < t_2 \leqslant t_+$, then $B(E_{t_1}) \subseteq B(E_{t_2})$.
In the absence of jump intervals, which follows from the $\det(H) \neq 0$ almost everywhere condition we imposed on $H(t)$, all such inclusions are isometric, see \cite{dB, Remling}. 

A positive measure $\mu$ on $\R$  is called the spectral measure of \eqref{CS} corresponding to a fixed self-adjoint boundary condition at $0$ if all de Branges spaces $B(E_t)$, $t \in (0, t_+)$ are isometrically embedded into $L^2(\mu)$. Under the assumptions we imposed on the system, such a spectral measure always exists, and it is unique if and only if \begin{equation*}
    \int_{0}^{t_+} \text{trace~} H(t) dt = \infty.
\end{equation*}
The case when the integral above is infinite is called the limit point case, otherwise it is called the limit circle case. 

We say that a positive measure $\mu$ on $\R$ is Poisson-finite if
$$\int_\R\frac{d \mu(x)}{1+x^2}<\infty.$$

Under the condition  $H(t) \in L^1_{loc}$, any spectral measure $\mu$ is Poisson-finite, see for instance \cite{Remling}. We are mostly interested in det-normalized canonical systems on the half line, in other words $t_+ = \infty$. Therefore, we are in the limit point case.

We pay special attention to the Neumann condition, $X(0, z) = \begin{pmatrix} 1 \\ 0 \end{pmatrix}$, and the Dirichlet condition, $X(0, z) = \begin{pmatrix} 0 \\ 1 \end{pmatrix}$. Throughout the paper, we use $E(t, z) = u(t, z) - i v(t, z)$ and $\mu$ ($\Tilde{E}(t, z) = \Tilde{u}(t, z) - i \Tilde{v}(t, z)$ and $\Tilde{\mu}$) to denote the Hermite-Biehler entire function and the spectral measure corresponding to the Neumann (Dirichlet) condition at $0$.

Instead of the two-dimensional vector differential equation system, one can also study the $2 \times 2$ matrix that solves \eqref{CS}. The matrix $M(t, z)$ satisfying the initial value condition $M(0, z) = \begin{pmatrix}1 & 0 \\ 0 & 1 \end{pmatrix}$ is called the \textit{transfer matrix} or the \textit{matrizant} of the system. Note that the first column of $M$ is the vector solution of \eqref{CS} satisfying the Neumann condition at $t = 0$, and the second column is the solution satisfying the Dirichlet condition. It follows from \eqref{CS} that $\det M(t, z) = 1$ for all $t$ and $z$.

\subsection{PW-measures and PW-systems}
This subsection is based on \cite{MP} and references therein.

\begin{definition}
A positive Poisson-finite measure $\mu$ is sampling for the Paley-Wiener space $PW_a$ if there exist constants $0 < c < C$ such that for all $f \in PW_a$ \begin{equation*}
    c \norm{f}_{PW_a} \leqslant \|f\|_{L^2(\mu)} \leqslant C \norm{f}_{PW_a}.
\end{equation*}
\end{definition}
It is not difficult to show that any PW-measure is Poisson-finite. For this and other properties of PW-measures, see \cite{MP} and references therein.

We call a canonical system on the half-line $\mathbb{R}_+$ a PW-system if the corresponding de Branges spaces $B(E_t) $ are equal to $PW_t$ as sets for all $t \in (0, \infty)$.
In this case the spectral measure $\mu$ of the system is a PW-measure. It follows that PW-systems can be equivalently defined as those systems whose spectral measures belong to the PW-class.

The class of PW-measures admits the following elementary description. 

\begin{definition}
Let $\mu$ be a positive measure on $\R$. We call an interval $I \subseteq \mathbb{R}$ a $(\mu, \delta)$-interval if \begin{equation*}
    \mu(I) > \delta, \quad \text{and} \quad \abs{I} > \delta,
\end{equation*}
where $\abs{I}$ stands for the length of the interval $I$.
\end{definition}

\begin{theorem}\label{PWcondition}\cite{MP}
A positive Poisson-finite measure $\mu$ is a Paley-Wiener measure if and only if \begin{itemize}
    \item $\sup_{x \in \mathbb{R}} (x, x + 1) < \infty$,
    \item For any $d > 0$, there exists $\delta > 0$ such that for all sufficiently large intervals $I$, there exist at least $d \abs{I}$ disjoint $(\mu, \delta)$-intervals intersecting $I$.
\end{itemize}
\end{theorem}

We say that a measure $\mu$ has locally infinite support if $\text{supp\ }\mu\cap [-C,C]$ is an infinite set for some $C > 0$, or equivalently if $\text{supp\ }\mu$ has a finite accumulation point.
For a periodic measure one can easily deduce the following result:

\begin{corollary}\label{PWper}\cite{MP}
A positive locally finite periodic measure is a Paley-Wiener measure if and only if it has locally infinite support.
\end{corollary}

\subsection{Spectral problems}

Direct spectral problems for differential operators aim at finding the spectral data of a given operator. Inverse spectral problems seek to recover a differential operator from known spectral data. We now review some results related to spectral problems.

Because of the following theorem, the prototype case in spectral problems is a diagonal Hamiltonian, or equivalently, an even spectral measure on the real line. This is the setting we consider from this point onward.

\begin{theorem}[\cite{MP}]
If the Hamiltonian $H$ is diagonal, then the spectral measure $\mu$ is even. Conversely, if $\mu$ is an even positive Poisson-finite measure, then there exists a canonical system with diagonal Hamiltonian $H$ whose spectral measure is $\mu$.
\end{theorem}

For convenience, we identify the $2\pi-$periodic measure $\mu$ on the real line $\mathbb{R}$ with the measure $\mu|_{[-\pi, \pi)}$ on the unit circle $\mathbb{T}$. In the case of a $2\pi-$periodic spectral measure, the Hamiltonian $H$ is a step function of step size $1/2$, and the value of the $n-$th step has the following connection to values of orthogonal polynomials on the unit circle. It is worth pointing out that while we only consider even spectral measures in this note, the following result does not assume that the spectral measure is even.

\begin{theorem}[\cite{MP}]\label{onpoly}
Let $\mu$ be a $2\pi-$periodic PW-measure, and $\{\varphi_n\}_{n = 0}^\infty$ be the family of polynomials orthonormal in $L^2_{\frac{\mu}{2\pi}}(\mathbb{T})$. We have \begin{equation*}
    h_{11}(t) |_{(\frac{n}{2}, \frac{n + 1}{2}]} = |\varphi_n(1)|^2.
\end{equation*}
\end{theorem}

\begin{remark}
If $\mu$ is a $2T-$periodic measure, we have \begin{equation*}
    h_{11}(t) |_{(n\frac{\pi}{2T}, (n + 1)\frac{\pi}{2T}]} = |\varphi_n(1)|^2,
\end{equation*} and the orthonormal polynomials on the unit circle $\mathbb{T}$ are defined using $d \mu^\ast (x) = d\mu\left( \frac{T}{\pi} x \right) $.
\end{remark}

By the results above, obtaining the spectral measure $\mu$ from a diagonal step-function Hamiltonian $H$ is equivalent to recovering the measure on the unit circle $\mathbb{T}$ from the sequence $ \left\{ \abs{\varphi_n(1)}^2 \right\}$. In other words, we will prove the following result to solve the direct spectral problem.

\begin{theorem} \label{Unique}
Let $\{h_{11}^n\}_{n = 0}^\infty$ be an arbitrary sequence of positive real numbers, there exists a unique even measure $\mu$ on the unit circle $\mathbb{T}$ such that \begin{equation*}
    \abs{\varphi_n(1)}^2 = h_{11}^n,
\end{equation*}
where $\{ \varphi_n \}$ is the family of polynomials orthonormal in $L^2_{\frac{\mu}{2\pi}}(\mathbb{T})$.
\end{theorem}

We will provide two proofs for this theorem: one by recovering the Verblunsky coefficients in section \ref{alpha}, and the other by recovering the moments in section \ref{Moments}. Both proofs are algebraic relying on known formulas for orthogonal polynomials on the unit circle.

\section{Recovering the Verblunsky coefficients}\label{alpha}
\subsection{Verblunsky coefficients and Szeg\"o recurrence}
This subsection is primarily based on chapter 1 in \cite{Simon} and references therein.

Let $P$ be a polynomial of degree $n$ on the unit circle $\mathbb{T}$, its \textit{reverse polynomial} is defined by $$P^\ast (z) = z^n \overline{P(1/\overline{z})}.$$

Let $\mu$ be a nontrivial probability measure on $\mathbb{T}$ , $\Phi_n(z)$ be the family of monic orthogonal polynomials on $L^2_{\frac{\mu}{2\pi}}(\mathbb{T})$, and $\Phi_n^\ast(z)$ be the reverse polynomials. $\Phi_n$ and $\Phi_n^\ast$ satisfy the following properties: \begin{enumerate}
\item $\left( \Phi_n^\ast \right)^\ast = \Phi_n$;
\item Up to a unique constant, $\Phi_n^\ast$ is the unique polynomial of degree (at most) $n$ that is orthogonal to $z,z^2, \ldots, z^n$.
\item $\norm{\Phi_n} = \norm{\Phi^\ast_n}$.
\end{enumerate}

We can define a sequence of numbers $\{ \alpha_n \}_{n = 0}^\infty \subseteq \mathbb{D}$ such that  
\begin{equation}\label{Szego}
\begin{split}
    \Phi_{n + 1}(z) &= z \Phi_n(z) - \overline{\alpha_n} \Phi^\ast_n(z), \\
    \Phi_{n + 1}^\ast(z) &= \Phi^\ast_n(z) - \alpha_n z \Phi_n(z). 
\end{split}    
\end{equation}
Moreover, \begin{equation*}
    \norm{\Phi_{n + 1}}^2 = \left( 1 - \abs{\alpha_n}^2 \right) \norm{\Phi_n}^2 = \prod_{j = 0}^n \left( 1 - \abs{\alpha_j}^2 \right).
\end{equation*}

The relations \eqref{Szego} are called the \textit{Szeg\"o recurrence} or \textit{Szeg\"o difference equations}, and $\alpha_n$'s are called the \textit{Verblunsky coefficients}. The Szeg\"o recurrence can also be rewritten for the orthonormal polynomials $\varphi_n$:
\begin{equation}\label{SzegoON}
    \varphi_{n + 1}(z) = \left( 1 - \abs{\alpha_n}^2 \right)^{-1/2} \left( z \varphi_n(z) - \overline{\alpha_n}\varphi_n^\ast(z) \right).
\end{equation}

In the second part of this section, we will solve for $\alpha_n$'s using $\abs{\varphi_n(1)}^2$'s. These $\alpha_n$'s uniquely identify a non-trivial even probability measure on $\mathbb{T}$ because of the Verblunsky's theorem.

\begin{theorem}[Verblunsky's Theorem]

Let $\{\alpha_j\}_{j = 0}^\infty$ be a sequence of numbers in $\mathbb{D}$. Then there is a unique probability measure $d\mu$ with $\alpha_j(d\mu) = \alpha_j$.
    
\end{theorem}

Since a $2\pi-$periodic spectral measure $\mu$ is not guaranteed to be a probability measure, discussing its Verblunsky coefficients may not always make sense. However, one can always consider the associated probability measure $$\nu := \frac{1}{\int_\mathbb{T} \frac{d\mu(t)}{2\pi}} \mu,$$ and find the Verblunsky coefficients of $\nu$ instead. When referring to the "Verblunsky coefficients" of $\mu$, we mean the Verblunsky coefficients of the probability measure $\nu$.

\subsection{Recovering the Verblunsky coefficients}

\begin{proof}[Proof of theorem \texorpdfstring{\ref{Unique}}{Unique}]

We first assume that $h_{11}^0 = 1$, which is the case when $\mu$ is a probability measure.

Note that when $\mu$ is even, both $\varphi_n(1)$ and $\alpha_n$ are real for every $n$. Since \begin{equation*}
    \varphi_n(z) = z^n \overline{\varphi_n^\ast\left( \frac{1}{\overline{z}} \right)},
\end{equation*}
we have $\varphi_n(1) = \varphi_n^\ast(1)$. Therefore, \eqref{SzegoON} becomes \begin{equation*}
    \varphi_{n + 1}(1) =  \left( 1 - \alpha_n^2 \right)^{-1/2}  \left( \varphi_n(1) - \alpha_n \varphi_n(1) \right),
\end{equation*}

which leads to \begin{equation}\label{SzegoUsed}
    \frac{\left( \varphi_{n + 1}(1) \right)^2}{\left( \varphi_{n}(1) \right)^2} = \frac{1 - \alpha_n}{1 + \alpha_n}.
\end{equation}

The left-hand side of the equation above is $h_{11}^{n + 1}/h_{11}^n$, which is a positive number. The function $f(x) = \frac{1 - x}{1 + x}$ is a bijection between $(-1, 1)$ and $(0, \infty)$. Therefore, we can use $\eqref{SzegoUsed}$ to recover the Verblunsky coefficients, and 
\begin{equation*}
    \alpha_n = \frac{1 - h_{11}^{n + 1}/h_{11}^n}{1 + h_{11}^{n + 1}/h_{11}^n}.
\end{equation*}

If $h_{11}^0 \neq 1$, we can still use the process above to recover a probability measure $\nu$ on the unit circle by computing the Verblunsky coefficients. The spectral measure $\mu$ of the corresponding system is $\cfrac{\nu}{h_{11}^0}$.
\end{proof}

\begin{remark}
Although the Verblunsky coefficients uniquely identify a probability measure on the unit circle, we would like to explicitly recover the measure under our setup of spectral problems. However, reconstructing the measure from these coefficients is not a trivial process. The usual methods involve obtaining the moments of the measure, denoted by $c_n$, either by applying the Verblunsky's formula \begin{equation*}
    c_{n + 1} = \alpha_n \prod_{j = 0}^{n - 1} \left( 1 - \alpha_j^2 \right) + V^{(n)}(\alpha_0, \alpha_1, \ldots, \alpha_{n - 1}),
\end{equation*}
where $V^{(n)}$ is a polynomial with integer coefficients, or by solving for $c_{n + 1}$ using a part of the Heine's formulas \begin{equation*}
    \alpha_n = (-1)^n \frac{\det(K_n)}{\det(J_n)},
\end{equation*}
where $J_n$ and $K_n$ are the truncated Toeplitz matrices (defined in \eqref{TruncatedToeplitz} below) of $\mu$ and $e^{-i\theta}\mu$, respectively. Thus, a natural question is whether it is possible to recover the moments directly from $\{h_{11}^n\}$. The answer is positive, and we provide this alternative algorithm in the next section.
\end{remark}

\section{Recovering the moments}\label{Moments}

\subsection{Moments and orthogonal polynomials}

This subsection is primarily based on chapter 1 in \cite{Simon} and references therein.

For a $2\pi-$periodic measure $\mu$, its $n-$th moment is defined by \begin{equation*}
    c_n = \int_0^{2\pi} e^{- i n x} \frac{d \mu(x)}{2\pi} = a_n + i b_n.
\end{equation*}

Associated with the moments is the infinite Toeplitz matrix \begin{equation*}
    J = \begin{bmatrix}
    c_0 & c_1 & c_2 & c_3 & \ldots \\
    c_{-1} & c_{0} & c_1 & c_2 & \ldots \\
    c_{-2} & c_{-1} & c_0 & c_1 & \ldots \\
    c_{-3} & c_{-2} & c_{-1} & c_0 & \ldots \\
    \vdots & \vdots & \vdots & \vdots & \ddots
    \end{bmatrix}.
\end{equation*}
We will denote by $J_n$ the $n \times n$ submatrix on the upper left corner of $J$. In other words \begin{equation}\label{TruncatedToeplitz}
    J_n = \begin{bmatrix}
    c_0 & c_1 & c_2 & \ldots & c_n \\
    c_{-1} & c_{0} & c_1 & \ldots & c_{n - 1} \\
    c_{-2} & c_{-1} & c_0 & \ldots & c_{n - 2} \\
    \vdots & \vdots & \vdots & \ddots & \vdots \\
    c_{-n} & c_{-n + 1} & c_{-n + 2} & \ldots & c_0
    \end{bmatrix}.
\end{equation}
Note that our matrix index starts from $0$ instead of $1$.

It is a well-known result (the Carath\'eodory-Toeplitz theorem) that $\mu$ is a finite positive $2\pi-$periodic measure with infinite support if and only if $$\det(J_n) > 0, \quad \forall n.$$

The Heine formulas give explicit relations between the moments and the monic orthogonal polynomials:\begin{equation*}
\begin{split}
    &\Phi_n(z) = \frac{1}{\det(J_{n - 1})}  \begin{vmatrix}
    c_0 & c_{-1} & c_{-2} & \ldots & c_{-     n} \\
    c_{1} & c_0 & c_{-1} & \ldots & c_{- n + 1}\\
    c_{2} & c_{1} & c_0 & \ldots & c_{- n + 2} \\
    \vdots & \vdots & \vdots & \ddots & \vdots \\
    1 & z & z^2 & \ldots & z^n
    \end{vmatrix},\\
    &\|\Phi_n(z)\|^2 = \frac{\det(J_{n})}{\det(J_{n - 1})}
\end{split}    
\end{equation*}

\subsection{Toeplitz determinants}

To prove theorem \ref{Unique} by directly recovering the moments of the measure, we need several algebraic results for Toeplitz determinants. 

First note that under the restriction of an even measure $\mu$, all the $J_n$ matrices are real and symmetric about both diagonals. 

For convenience, we will index vector lengths from $0$ instead of $1$, as we do for matrices. In particular, we consider $J_0 = [c_0]$ as a $0 \times 0$ matrix and $v_1 = (c_1)^T$ as a vector of length $0$.

\begin{lemma}
If $J_n$ and $J_{n - 1}$ are truncated Toeplitz matrices of $J$, then
\begin{equation}\label{detformula1}
    \det (J_n) = \det(J_{n - 1}) \left( c_0 - v_n^T J_{n - 1}^{-1} v_n \right),
\end{equation}
where $v_n = (c_n, c_{n - 1}, \ldots, c_1)^T$.
\end{lemma}

\begin{proof}
We denote by $M_{i, j}^n$ the $(i, j)$-minor of $J_n$.

The matrix $J_n$ can be rewritten as the block matrix $\begin{bmatrix} J_{n - 1} & v_n \\v_n^T & c_0\end{bmatrix}$. We first compute $\det(J_n)$ using a cofactor expansion along the last column, \begin{equation}\label{det1}
    \det(J_n) = c_0 \det(J_{n - 1}) + \sum_{i = 1}^n (-1)^i c_i \det \left(M_{n - i, n}^n \right).
\end{equation} 

Note that all the $M_{n - i, n}^n$'s, $i = 1, 2, \ldots, n$ share the same $(n - 1)$st row (the $n-$th row of $J_n$). We compute $\det(M_{n - i, n}^n)$ by using the cofactor expansion along the last row, \begin{equation}\label{det2}
    \det \left( M_{n - i, n}^n \right) = \sum_{j = 1}^n (-1)^{j + 1} c_j \det \left( M_{n - i, n - j}^{n - 1} \right).
\end{equation}

Combining \eqref{det1} and \eqref{det2}, \begin{equation}\label{det3}
    \det(J_n) = c_0 \det(J_{n - 1}) + \sum_{i = 1}^n \sum_{j = 1}^n (-1)^{i + j + 1} c_i c_j \det \left( M_{n - i, n - j}^{n - 1} \right).
\end{equation}

The inverse of $J_{n - 1}$ is \begin{equation*}
    \frac{1}{\det (J_{n - 1})} \left[ (-1)^{i + j} \det \left( M_{i, j}^{n - 1} \right) \right]^T.
\end{equation*}

We now compute $v_n^T J_{n - 1}^{-1} v_n$: \begin{equation*}
\begin{split}
    v_n^T J_{n - 1}^{-1} v_n &= \frac{1}{\det(J_n)} \begin{pmatrix} c_n & c_{n - 1} & \ldots & c_1 \end{pmatrix} \left[ (-1)^{i + j} \det \left( M_{i, j}^{n - 1} \right) \right]^T \begin{pmatrix} c_n \\ c_{n - 1} \\ \vdots \\ c_1 \end{pmatrix} \\
    &= \frac{1}{\det(J_n)} \sum_{l = 0}^{n - 1} \sum_{k = 0}^{n - 1} (-1)^{l + k} c_{n - l} c_{n - k} \det \left( M^{n - 1}_{l, k} \right).
\end{split} 
\end{equation*}

Reindex the equation above with $i = n - l$ and $j = n - k$, \begin{equation*}
\begin{split}
    v_n^T J_{n - 1}^{-1} v_n &= \frac{1}{\det(J_n)} \sum_{i = n}^1 \sum_{j = n}^1 (-1)^{2n - i - j} c_i c_j \det \left( M_{n - i, n - j}^{n - 1}\right)\\
    &= \frac{1}{\det(J_n)} \sum_{i = 1}^n \sum_{j = 1}^n (-1)^{i + j} c_i c_j \det \left( M_{n - i, n - j}^{n - 1}\right).
\end{split}    
\end{equation*}

\eqref{detformula1} is obtained by combining the equation above with \eqref{det1}.
\end{proof}

By a similar calculation, we obtain the following.
\begin{corollary}
\begin{equation}\label{detformula2}
    \det \begin{pmatrix}
        J_{n - 1} & v_n \\ \mathbb{1} & 1
    \end{pmatrix} = \det (J_{n - 1}) \left( 1 - \mathbb{1} J_{n - 1}^{-1} v_n \right),
\end{equation}
where $v_n = (c_n, c_{n - 1}, \ldots. c_1)^T$, and $\mathbb{1} = (1, 1, \ldots, 1)$ is a vector of length $n - 1$.
\end{corollary}

When $c_0 = 1$, the matrix $J_n$ can also be written in the form \begin{equation*}
J_n = \begin{bmatrix} 1 & u_n^T \\ u_n & J_{n - 1} \end{bmatrix},
\end{equation*}
where $u_n = (c_1, c_2, \ldots, c_n)^T$. In this case, we have the following formula for expressing $J_{n}^{-1}$ in terms of $J_{n-1}^{-1}$ and $u_n$.
\begin{theorem}[\cite{Trench}]
\begin{equation}\label{ToeplitzInverseFormula}
    J_n^{-1} = \begin{bmatrix} \Delta_n^{-1} & & - \Delta_n^{-1} u_n^T J_{n - 1}^{-1} \\[1em] - \Delta_n^{-1} J_{n - 1}^{-1} u_n & & J_{n - 1}^{-1} + \Delta_n^{-1} J_{n - 1}^{-1} u_n u_n^T J_{n - 1}^{-1} \end{bmatrix},
\end{equation}
where $\Delta_n = 1 - u_n^T J_{n - 1}^{-1} u_n$.
\end{theorem}

\subsection{Recovering the moments}

\begin{proof}[Proof of theorem \ref{Unique}]
We first assume that $c_0 = 1$, which is $h_{11}^0 = \abs{\varphi_0(1)}^2 = 1$. We solve for $c_1$ separately: \begin{equation*}
    h_{11}^1 = \abs{\varphi_1(1)}^2 = \frac{\begin{vmatrix} 1 & c_1 \\ 1 & 1 \end{vmatrix}^2}{\begin{vmatrix} 1 & c_1 \\ c_1 & 1 \end{vmatrix}} = \frac{(1 - c_1)^2}{1 - c_1^2} = \frac{1 - c_1}{1 + c_1},
    \end{equation*}
    has a unique solution $$c_1 = \frac{1 - h_{11}^1}{1 + h_{11}^1}.$$
For $n > 1$, Heine formulas together with \eqref{detformula1} and \eqref{detformula2} yield \begin{equation} \label{QuadraticFraction}
        h_{11}^{n + 1} = \varphi_{n + 1}(1)^2 = \frac{\left( 1 - \mathbb{1} J_{n}^{-1} v_{n + 1} \right)^2}{1 - v_{n + 1}^T J_{n}^{-1} v_{n + 1} },
    \end{equation}
which is a quadratic fraction in $c_{n + 1}$. We now show that the numerator and denominator of this fraction share a common factor, and therefore, there exists a unique $c_{n + 1}$ satisfying \eqref{QuadraticFraction}.

Note that the sums of the $(i - 1)$st and $(n - i)$th rows or columns of $J_{n - 1}^{-1}$ are the same since $J_{n - 1}^{-1}$ is symmetric about both diagonals. Using \eqref{ToeplitzInverseFormula}, \begin{equation*}
    \begin{split}
        1 - \mathbb{1}J_n^{-1} v_{n + 1} &= 1 - \mathbb{1} \begin{bmatrix} \Delta_n^{-1} & & - \Delta_n^{-1} u_n^T J_{n - 1}^{-1} \\[1em] - \Delta_n^{-1} J_{n - 1}^{-1} u_n & & J_{n - 1}^{-1} + \Delta_n^{-1} J_{n - 1}^{-1} u_n u_n^T J_{n - 1}^{-1} \end{bmatrix} \begin{bmatrix} c_{n + 1} \\ v_n \end{bmatrix} \\
        &= \left(-\Delta_n^{-1} c_{n + 1} + 1 + \Delta_n^{-1} u_n^T J_{n - 1}^{-1} v_n \right) \left( 1 - \mathbb{1} J_{n - 1}^{-1} u_n \right).
    \end{split}
    \end{equation*}
    
Since $\left( 1 - \mathbb{1} J_{n - 1}^{-1} u_n \right)$ is a constant with respect to $c_{n + 1}$, it suffices to show that $$\left(-\Delta_n^{-1} c_{n + 1} + 1 + \Delta_n^{-1} u_n^T J_{n - 1}^{-1} v_n \right)$$ is a factor of $ 1 - v_{n + 1}^T J_{n }^{-1} v_{n + 1} $. Again using \eqref{ToeplitzInverseFormula}, \begin{equation*}
\begin{split}
    1 - v_{n + 1}^T J_{n }^{-1} v_{n + 1} &= 1 - \begin{bmatrix} c_{n + 1} & v_n^T \end{bmatrix} \begin{bmatrix} \Delta_n^{-1} & & - \Delta_n^{-1} u_n^T J_{n - 1}^{-1} \\[1em] - \Delta_n^{-1} J_{n - 1}^{-1} u_n & & J_{n - 1}^{-1} + \Delta_n^{-1} J_{n - 1}^{-1} u_n u_n^T J_{n - 1}^{-1} \end{bmatrix} \begin{bmatrix} c_{n + 1} \\ v_n \end{bmatrix} \\
    &= -\Delta_n^{-1} c_{n + 1}^2 + 2 \Delta_n^{-1} u_n^T J_{n - 1}^{-1} v_n c_{n + 1} + 1 - v_n^T J_{n - 1}^{-1} v_n - \Delta_n^{-1} v_n^T J_{n - 1}^{-1} u_n u_n^T J_{n - 1}^{-1} v_n.        
\end{split}
\end{equation*}

If
$$\left(-\Delta_n^{-1} c_{n + 1} + 1 + \Delta_n^{-1} u_n^T J_{n - 1}^{-1} v_n \right)$$
is a linear factor of $1 - v_{n + 1}^T J_{n }^{-1} v_{n + 1}$, then the other linear factor is \begin{equation*}
    c_{n + 1} + \frac{1 - v_n^T J_{n - 1}^{-1} v_n - \Delta_n^{-1} v_n^T J_{n - 1}^{-1} u_n u_n^T J_{n - 1}^{-1} v_n}{1 + \Delta_n^{-1} u_n^T J_{n - 1}^{-1} v_n}.
\end{equation*}

Recall that $\Delta_n = 1 - u_n^T J_{n - 1} u_n = 1 - v_n^T J_{n - 1} v_n$. We simplify the second term in the expression above: \begin{equation*}
\begin{split}
    \frac{1 - v_n^T J_{n - 1}^{-1} v_n - \Delta_n^{-1} v_n^T J_{n - 1}^{-1} u_n u_n^T J_{n - 1}^{-1} v_n}{1 + \Delta_n^{-1} u_n^T J_{n - 1}^{-1} v_n} &= \frac{ \Delta_n - \Delta_n^{-1} v_n^T J_{n - 1}^{-1} u_n u_n^T J_{n - 1}^{-1} v_n}{1 + \Delta_n^{-1} u_n^T J_{n - 1}^{-1} v_n}\\
    &= \Delta_n \frac{1 - \Delta_n^{-2} (v_n^T J_{n - 1} u_n)^2}{1 + \Delta_n^{-1} u_n^T J_{n - 1}^{-1} v_n} = \Delta_n \left( 1 - \Delta_n^{-1} v_n^T J_{n - 1}^{-1} u_n \right).
\end{split}
\end{equation*}

We now have that \eqref{QuadraticFraction} is actually a linear fraction since \begin{equation*}
    \left(-\Delta_n^{-1} c_{n + 1} + 1 + \Delta_n^{-1} u_n^T J_{n - 1}^{-1} v_n \right) \left(c_{n + 1}  + \Delta_n - v_n^T J_{n - 1}^{-1} u_n \right) = 1 - v_{n + 1}^T J_{n }^{-1} v_{n + 1}.
\end{equation*}   

Finally, when $h_{11}^0 \neq 1$, we have $$c_0 = \frac{1}{h_{11}^0}.$$ Following the process above, we can solve for $\Tilde{c}_n$ using $\Tilde{h}^n = h_{11}^n/h_{11}^0$, and the desired $c_n$ values are given by $c_n = c_0 \cdot \Tilde{c}_n$.
\end{proof}

\begin{remark}
For $n > 1$, we can use the following formula to solve for $c_{n + 1}$ using the first $n$ moments: \begin{equation}\label{Algorithm}
    c_{n + 1} = \frac{1}{h_{11}^0} \cdot \frac{\left( 1 + D_n/\Delta_n \right) \left( 1 - \mathbb{1}J_{n - 1}^{-1} u_n \right)^2 - (\Delta_n - D_n) h_{11}^{n + 1}}{h_{11}^{n + 1} \Delta_n + \left( 1 - \mathbb{1}J_{n - 1}^{-1} u_n \right)^2/\Delta_n},
\end{equation}
where $\Delta_n = 1 - u_n^T J_{n - 1}^{-1} u_n$ and $D_n = u_n^T J_{n - 1}^{-1} v_n$. This formula is used in the examples in section \ref{examples}.
\end{remark}

\section{Beyond periodic spectral measures}\label{Approx}

Let $H$ be a det-normalized diagonal Hamiltonian of a canonical system. In \cite{MP}, it is shown that if $H$ is a step function of uniform step size, then the corresponding spectral measure $\mu$ is an even periodic PW-measure. Moreover, we can solve this direct spectral problem now by obtaining the moments and/or the Verblunsky coefficients of the measure. 

If $H = \begin{pmatrix} h_{11} & 0 \\ 0 & h_{22}\end{pmatrix}$ is not a step function, we can 'periodize' the Hamiltonian by defining a step-function Hamiltonian $H^T$ as follows: \begin{equation}\label{periodization}
\begin{split}
    h_{11}^{T, n} &= h_{11}^T \big|_{\left[ nT, (n + 1)T \right)]} = \frac{1}{T} \int_{nT}^{(n + 1)T} h_{11}(s) ds,\\
    h_{22}^{T, n} &= h_{22}^T \big|_{\left[ nT, (n + 1)T \right)]} = \frac{1}{h_{11}^{T, n}}.
\end{split}
\end{equation}
The spectral measure corresponding to the canonical system with Hamiltonian $H^T$ is denoted by $\mu_T$.

Given that all $\mu_T$'s are PW-measures, our focus is on applying this approach to PW-systems, since all their de Branges spaces $B(E^T_t)$ are the same as $PW_t$ as sets. Unfortunately, characterizing PW-systems solely through their Hamiltonians is challenging when the entries of $H$ are not step functions. As discussed in \cite{MP}, describing PW-Hamiltonians that are not step functions, even in the diagonal case, is inherently difficult. It is known that not all Hamiltonians with non-vanishing determinants are PW-Hamiltonians. However, PW-Hamiltonians include those arising from real Dirac systems, as shown in several results in \cite{MP, BR}.

Real Dirac systems on the half-line $\mathbb{R}_+$ are of the form \begin{equation}\label{RD}
    \Omega \Dot{X}(t) = z X(t) - Q(t) X(t), \ \ t \in (0, \infty).
\end{equation}
Here, $z\in \mathbb{C}$ is a spectral parameter,
$\Omega = \begin{pmatrix} 0 & 1 \\ -1 & 0 \end{pmatrix}$ is the same symplectic matrix as in canonical systems, and $Q(t) = \begin{pmatrix} 0 & f(t) \\ f(t) & 0 \end{pmatrix}$ is the potential matrix, where $f(t)$ is a real-valued locally integrable function, and $X(t) = X(t, z) =  \begin{pmatrix} u(t, z) \\ v(t, z) \end{pmatrix}$ is the unknown vector-function. These systems, when rewritten as canonical systems, feature det-normalized diagonal Hamiltonians, with
$$h_{11}(t) = \exp \left( \int_0^t f(s) ds \right),$$
where $f$ is the locally integrable function in $Q(t)$. Details on rewriting such a system as a canonical system can be found, for instance, in \cite{MPS, Rom, Zhang}. In the rest of this paper, we will tackle the direct spectral problems for real Dirac systems.

When using the periodization approach on Hamiltonians from real Dirac systems, $H^T$'s will still be in the form of \begin{equation*}
    \begin{pmatrix}
        \exp \left( \int_0^t f^T(s) ds \right) & 0 \\ 0 & \exp \left( - \int_0^t f^T(s) ds \right)
    \end{pmatrix},
\end{equation*}
but with $f^T$ being an evenly spaced discrete measure. By the way these approximations are defined, as $T \to \infty$, we have \begin{equation*}
    \int_0^t \abs{f^T(s)} ds \to \int_0^t \abs{f(s)} ds,
\end{equation*}
as $T \to 0$.

We now revisit the Hermite-Biehler entire functions $E$ and $\Tilde{E}$. They satisfy the following differential equation derived from \eqref{RD}:
\begin{equation}\label{HBDFQ}
    \frac{\partial}{\partial t} E(t, z) = -iz E(t, z) + f(t) E^\# (t, z),
\end{equation}
with initial value conditions $E(0, z) = 1$ and $\Tilde{E}(0, z) = -i$, respectively.

Closely associated with the Hermite-Biehler functions are the scattering functions \begin{equation*}
    \mathcal{E}(t, z) = e^{itz} E(t, z) \; \text{and} \; \Tilde{\mathcal{E}}(t, z) = e^{itz} \Tilde{E} (t, z).
\end{equation*}
The scattering functions satisfy \begin{equation*}
    \frac{\partial}{\partial t} \mathcal{E}(t, z) = iz e^{itz} E(t, z) + e^{itz} \frac{\partial}{\partial t} E(t, z) =  f(t) e^{2itz} \mathcal{E}^\#(t, z),
\end{equation*}
with initial value conditions $\mathcal{E}(0, z) = 1$ and $\Tilde{\mathcal{E}}(0, z) = -i$, respectively.

Restrict the equation above to the real line, we get \begin{equation}\label{ScatterEQN}
    \frac{\partial}{\partial t} \mathcal{E}(t, x) = f(t) e^{2itx} \mathcal{E}^\#(t, x) = f(t) e^{2itx} \overline{\mathcal{E}(t, x)}.
\end{equation}

This implies \begin{equation*}
    \abs{\mathcal{E}(a, x)}, \abs{\Tilde{\mathcal{E}}(a, x)} \leqslant \exp \left( \int_0^a \abs{f(t)} dt \right),
\end{equation*}
which is equivalent to \begin{equation*}
    \abs{E(a, x)}, \abs{\Tilde{E}(a, x)} \leqslant \exp \left( \int_0^a \abs{f(t)} dt \right).
\end{equation*}
Together with $\det(M) = 1$ (see section \ref{CSdB}), this implies that $\abs{E}$ and $\abs{\Tilde{E}}$ are both bounded away from $0$ on $\mathbb{R}$.

Now we are ready to prove the main result in this section.

\begin{theorem}
Let $H$ be the Hamiltonian of a real Dirac system, and let $H^T$'s be the step-function approximations defined in \eqref{periodization}. Denote by $\mu$ the spectral measure of the real Dirac system, and $\mu_T$ the spectral measure of the canonical system with Hamiltonian $H^T$. Then, for $\phi \in PW_a$, $a > 0$, we have $\|\phi\|_{L^2(\mu_T)} \to \|\phi\|_{L^2(\mu)}$ as $T \to 0$.
\end{theorem}

\begin{proof}
First note that \begin{equation*}
\begin{split}
    &\norm{\phi}_{L^2(\mu)}^2 = \int_\mathbb{R} \phi(x) \overline{\phi(x)} \frac{dx}{|E(a, x)^2|}, \\
    &\text{and~~} \norm{\phi}_{L^2(\mu^T)}^2 = \int_\mathbb{R} \phi(x) \overline{\phi(x)} \frac{dx}{|E^T(a, x)^2|}.
\end{split}
\end{equation*}
Then by H\"older's inequality, it is enough to show that \begin{equation*}
    \frac{1}{\abs{E^T(a, x)^2}} \to \frac{1}{\abs{E(a, x)^2}}, \quad \text{uniformly~in~} x \text{~as~} T \to 0.
\end{equation*}
Since $\abs{E(a, x)^2}$ is bounded away from $0$, it suffices to show that \begin{equation*}
    \abs{E^T(a, x)^2}  \to \abs{E(a, x)^2}, \quad \text{uniformly~in~} x \text{~as~} T \to 0,
\end{equation*}
which is equivalent to \begin{equation*}
    \abs{\mathcal{E}^T(a, x)^2} \to \abs{\mathcal{E}(a, x)^2}, \quad \text{uniformly~in~} x \text{~as~} T \to 0.
\end{equation*}

Recall that \begin{equation*}
    \abs{\mathcal{E}(a, x)}^2 \leqslant \exp \left( 2 \int_0^a \abs{f(t)} dt \right).
\end{equation*}
Then we have for $T$ small enough, \begin{equation*}
    \abs{\mathcal{E}^T(a, x)}^2 \leqslant \exp \left( 2 \int_0^a \abs{f^T(t)} dt \right) \leqslant \exp \left( 4 \int_0^a \abs{f(t)} dt\right).
\end{equation*}

Since $\mathcal{E}$ and $\mathcal{E}^T$ satisfy \eqref{ScatterEQN} with $f$ and $f^T$ respectively, we have \begin{equation*}
\begin{split}
    \frac{\partial}{\partial t} \left[ \mathcal{E}(t, x) - \mathcal{E}^T(t, x) \right] &= f(t) e^{2itx} \overline{\mathcal{E}(t, x)} - f^T(t) e^{2itx} \overline{\mathcal{E}^T(t, x)} \\
    &= f(t) e^{2itx} \overline{\left[ \mathcal{E}(t, x) - \mathcal{E}^T(t, x) \right]} - \left[f^T(t) - f(t) \right] e^{2itx} \overline{\mathcal{E}^T(t, x)}, \\
    \mathcal{E}(0, x) - \mathcal{E}^T(0, x) &= 0.
\end{split} 
\end{equation*}

Multiply $\mathcal{E}(t, x) - \mathcal{E}^T(t, x)$ to both sides of the differential equation above, we get \begin{equation*}
\begin{split}
    \frac{1}{2} &\frac{\partial}{\partial t} \left[ \mathcal{E}(t, x) - \mathcal{E}^T(t, x) \right]^2 \\
    &= f(t) e^{2itx} \abs{\left[ \mathcal{E}(t, x) - \mathcal{E}^T(t, x) \right]}^2 - \left[f^T(t) - f(t) \right] e^{2itx} \overline{\mathcal{E}^T(t, x)}\left[ \mathcal{E}(t, x) - \mathcal{E}^T(t, x) \right],
\end{split}
\end{equation*}
which implies \begin{equation*}
    \frac{1}{2} \abs{\frac{\partial}{\partial t} \left[ \mathcal{E}(t, x) - \mathcal{E}^T(t, x) \right]^2} \leqslant \abs{f(t)} \; \abs{\left[ \mathcal{E}(t, x) - \mathcal{E}^T(t, x) \right]^2} + C_a \abs{f^T(t) - f(t)},
\end{equation*}
where $$
C_a = \exp \left( 2 \int_0^a \abs{f(t)} dt\right) \left[\exp \left( 2 \int_0^a \abs{f(t)} dt\right) + \exp \left( \int_0^a \abs{f(t)} dt\right) \right].
$$
Therefore, \begin{equation*}    
\abs{\left[ \mathcal{E}(a, x) - \mathcal{E}^T(a, x) \right]}^2 \leqslant 2 C_a \exp \left( 2 \int_0^a \abs{f(t)} dt \right)  \int_0^a \abs{f^T(t) - f(t)} dt \to 0 \quad \text{as~} T \to 0,
\end{equation*}
and this convergence is uniform in $x$.
\end{proof}

\section{Examples of direct spectral problems with periodic spectral measures}\label{examples}

We first look at step-function Hamiltonians. We start by revisiting an example studied in \cite{MP}.

\begin{example}\label{example1+cos}
Let us solve the direct spectral problem where $h_{11} = \frac{(n + 1)(n + 2)}{2}$ on $(\frac{n}{2}, \frac{n + 1}{2}]$, $h_{22}(t) = \frac{1}{h_{11}(t)}$, and $h_{12}(t) = h_{21}(t) = 0$. From section 7 of \cite{MP}, the spectral measure is $d\mu(x) = \left( 1 - \cos(x) \right) dx$. Therefore, we expect to recover $c_0 = 1$, $c_1 = -\frac{1}{2}$, and $c_n = 0$ for $n \geqslant 2$, or equivalently $\alpha_n = -\frac{1}{n + 2}$.

We can recover the Verblunsky coefficients using section \ref{alpha}.\\

{\underline{$n = 0$.}} \begin{equation*}
    \alpha_0 = \frac{1 - h_{11}^1/h_{11}^0}{1 + h_{11}^1/h_{11}^0} = -\frac{1}{2}.
\end{equation*}

{\underline{$n = 1$.}}\begin{equation*}
   \alpha_1 = \frac{1 - h_{11}^2/h_{11}^1}{1 + h_{11}^2/h_{11}^1} = -\frac{1}{3}.
\end{equation*}

{\underline{$n = 2$.}}\begin{equation*}
    \alpha_2 = \frac{1 - {h_{11}^3}/{h_{11}^2}}{1 + {h_{11}^3}/{h_{11}^2}} = -\frac{1}{4}.
\end{equation*}

{\underline{$n = 3$.}}\begin{equation*}
    \alpha_2 = \frac{1 - {h_{11}^4}/{h_{11}^3}}{1 + {h_{11}^4}/{h_{11}^3}} = -\frac{1}{5}.
\end{equation*}

We can also recover the moments using section \ref{Moments}.\\

{\underline{$n = 0$.}} \begin{equation*}
   c_0 = \frac{1}{h_{11}^0} = 1.
\end{equation*}

{\underline{$n = 1$.}}\begin{equation*}
    c_1 = \frac{1 - h_{11}^1}{1 + h_{11}^1} = \frac{1 - 3}{1 + 3} = \frac{-1}{2}.
\end{equation*}

{\underline{$n = 2$.}}\begin{equation*}
    c_2 = \frac{\left( 1 + D_1/\Delta_1 \right) \left( 1 - \mathbb{1}J_{0}^{-1} u_1 \right)^2 - (\Delta_1 - D_1) h_{11}^{2}}{h_{11}^{2} \Delta_1 + \left( 1 - \mathbb{1}J_{0}^{-1} u_1 \right)^2/\Delta_1} = \frac{\left(1 + \frac{1/4}{3/4} \right) \cdot \left( \frac{3}{2} \right)^2 - \left(\frac{3}{4} - \frac{1}{4} \right) \cdot 6 }{6 \cdot \frac{3}{4} + \left( \frac{3}{2} \right)^2 / \frac{3}{4}} = 0.
\end{equation*}

{\underline{$n = 3$.}}\begin{equation*}
    c_3 = \frac{\left( 1 + D_2/\Delta_2 \right) \left( 1 - \mathbb{1}J_{1}^{-1} u_2 \right)^2 - (\Delta_2 - D_2) h_{11}^{3}}{h_{11}^{3} \Delta_2 + \left( 1 - \mathbb{1}J_{1}^{-1} u_2 \right)^2/\Delta_2} =
    \frac{\left(1 + \frac{1/6}{2/3} \right) \cdot \left( 2 \right)^2 - \left(\frac{2}{3} - \frac{1}{6} \right) \cdot 10 }{10 \cdot \frac{2}{3} + \left( 2 \right)^2/\frac{2}{3}} = 0.
\end{equation*}
\end{example}

\begin{example}\label{GeometricGrowth}
Let us solve the direct spectral problem where $h_{11} = a^n$ on $(\frac{n}{2}, \frac{n + 1}{2}]$ where $a > 0$, $h_{22}(t) = \frac{1}{h_{11}(t)}$, and $h_{12}(t) = h_{21}(t) = 0$.

For this example, we will only recover the Verblunsky coefficients. Since $h_{11}^0 = 1$, the spectral measure is a probability measure on $\mathbb{T}$. Since \begin{equation*}
    \frac{h_{11}^{n + 1}}{h_{11}^n} = a , \quad \text{for~all~}n,
\end{equation*}
the corresponding measure has constant Verblunsky coefficients \begin{equation*}
    \alpha_n = \alpha \equiv \frac{1 - a}{1 + a} \quad \text{for~all~}n.
\end{equation*}
This measure corresponds to the Geronimus polynomials. Details about these polynomials can be found in for instance \cite{Golinskii}. Note that our notation is slightly different, their constant Verblunsky coefficient is our $-\alpha$. The corresponding measure on the circle has an absolutely continuous part \begin{equation*}
    w(x) = \frac{1}{\abs{1 + \alpha}} \frac{\sqrt{1 - \alpha^2 - \cos^2 \frac{x}{2}}}{\sin \frac{x}{2}},
\end{equation*}
supported on $\left[ 2\arcsin{\abs{\alpha}}, 2\pi - 2\arcsin{\abs{\alpha}} \right]$, and a singular part \begin{equation*}
    \frac{2}{\abs{1 + \alpha}^2} \left( \abs{\alpha + \frac{1}{2}}^2 - \frac{1}{4} \right) \delta_{z = 1},
\end{equation*}
if $\alpha > 0$, or equivalently $a < 1$.
\end{example}

For non-step-function Hamiltonians, we again start by revisiting an example studied in \cite{MP}. 

\begin{example}
Let $H$ be the diagonal det-normalized Hamiltonian with $h_{11}(t) = \frac{1}{(1 + \frac{1}{4}t)^2}$. The spectral measure of the corresponding system is $\mu = m + \frac{1}{4} \pi \delta_0$, where $m$ is the Lebesgue measure and $\delta_0$ is the unit point mass at $0$.

For each step-function approximation $h_{11}^T$, we can compute the moments and/or the "Verblunsky coefficients" of the spectral measure $d\mu^T(x) = w^T(x) dx$, where $w^T(x)$ is a cosine series. In the following figures, the yellow curves represent the partial sum of the first $20$ terms of $w^T(x)$, while the blue curves depict the first $20$ terms of the Fourier series of a $\frac{\pi}{T}-$periodization of $\mu$. The $\frac{\pi}{T}-$periodization of $\mu$ is the same as in \cite{PZ}, where $\mu_{\frac{\pi}{T}}(S) = \mu(S)$ for $S \subseteq \left[ -\frac{\pi}{2T}, \frac{\pi}{2T} \right)$, and extended periodically with period ${\frac{\pi}{T}}$.  Each figure below displays three periods of the measures.
\begin{figure}[H]
\centering
\includegraphics[width=0.345\textwidth]{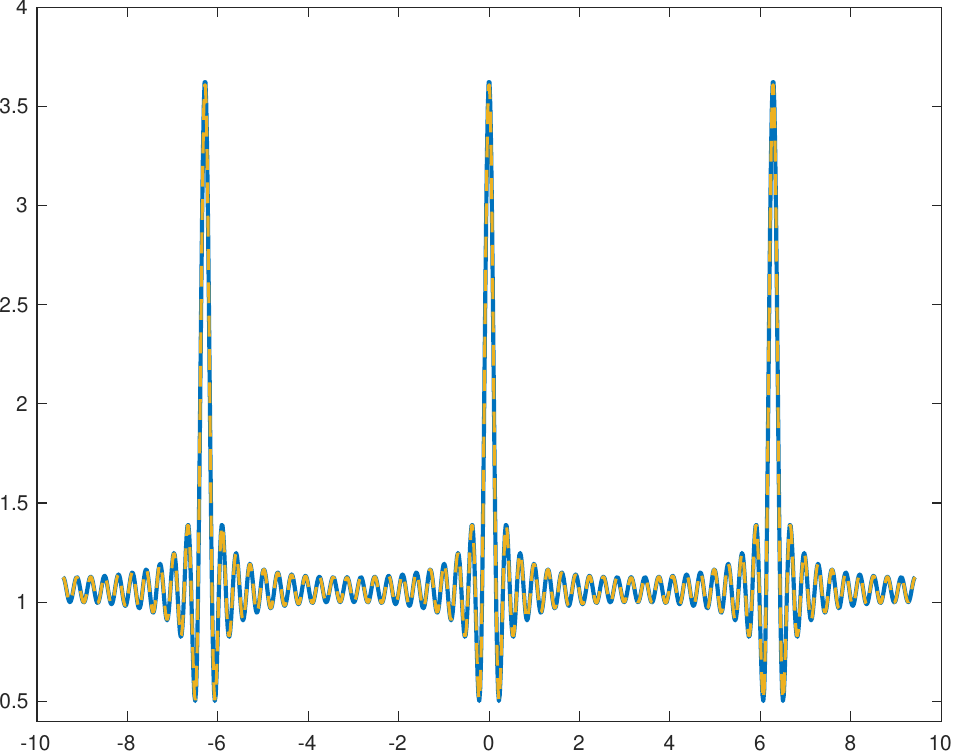}
\includegraphics[width=0.345\textwidth]{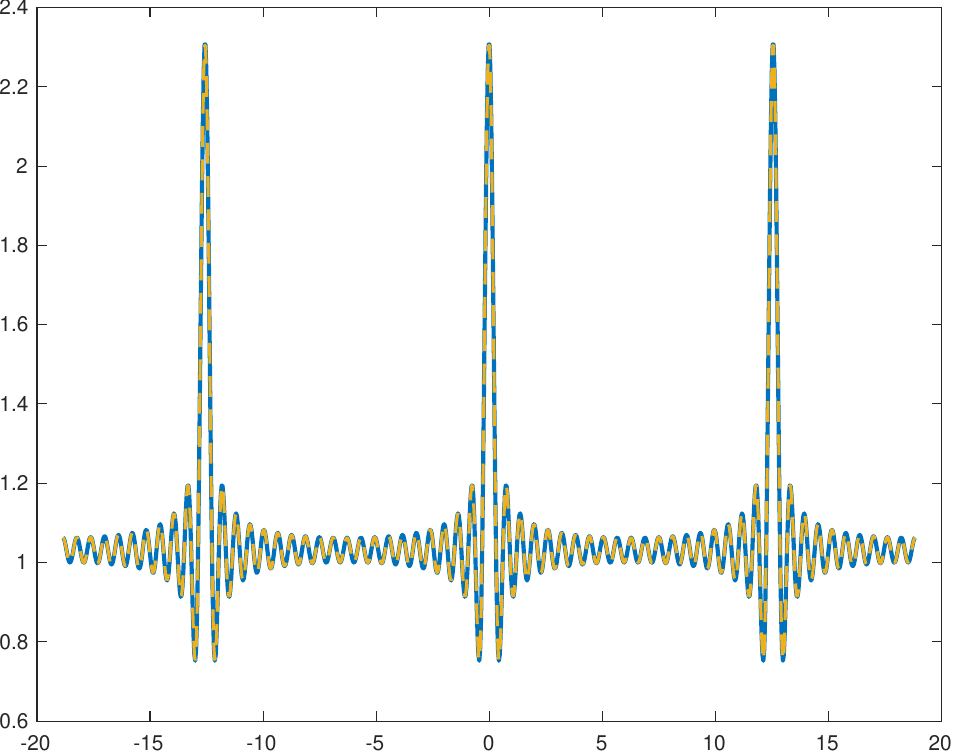}
\includegraphics[width=0.345\textwidth]{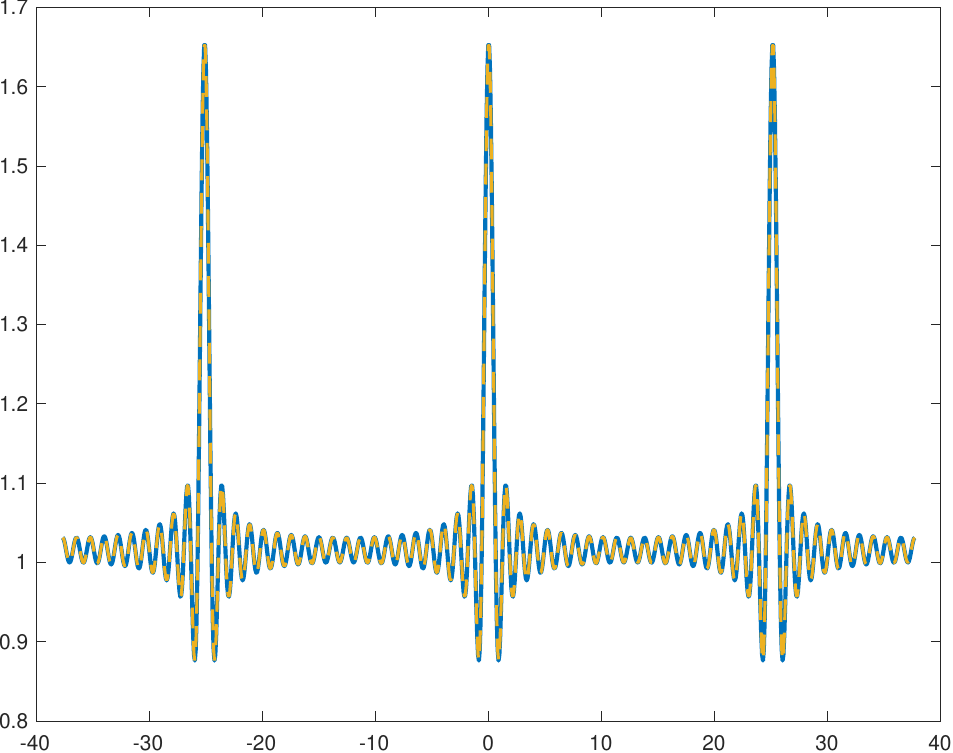}
\includegraphics[width=0.345\textwidth]{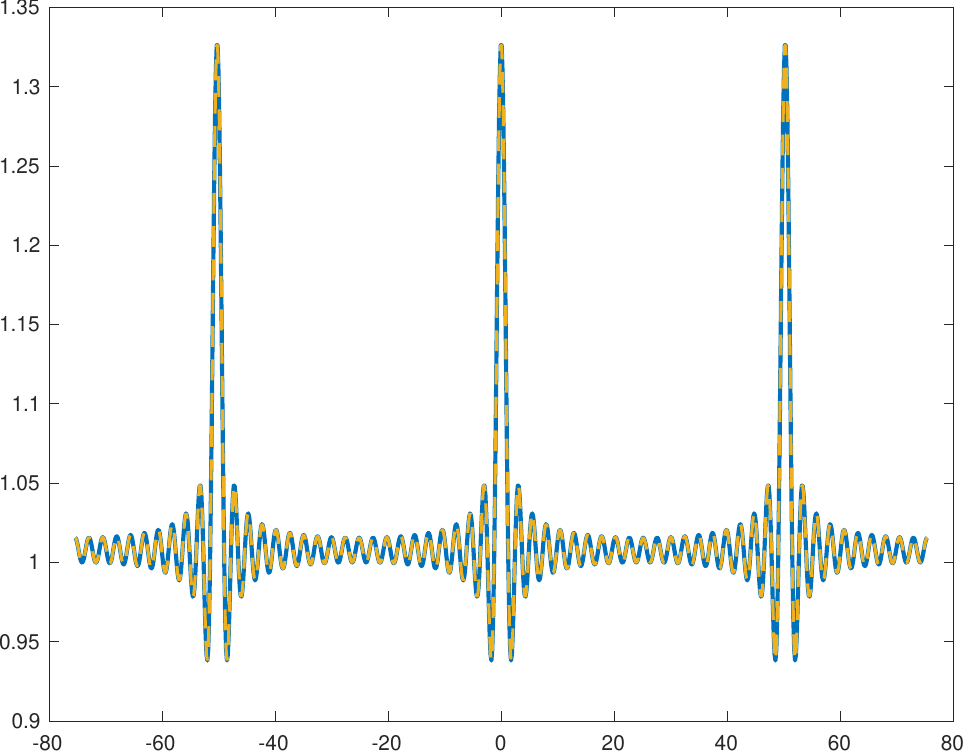}
\caption{First $20$ terms of the cosine series of $\mu^T$ and $\mu$. \\ 
Upper left: $T = \frac{1}{2}$, upper right: $T = \frac{1}{4}$, lower left: $T = \frac{1}{8}$, lower right: $T = \frac{1}{16}$.}
\end{figure}

\end{example}

Now we look at some new examples.

\begin{example}

Let $H$ be the diagonal det-normalized Hamiltonian with $h_{11}(t) = 1 + t$. For each step-function approximation $h_{11}^T$, we can compute the moments of the spectral measure $d\mu^T(x) = w^T(x) dx$, where $w^T(x)$ is a cosine series. In the following figure, the curves represent the partial sums of $w^T(x)$. 

\begin{figure}[H]
    \centering
    \includegraphics[width = .45\textwidth]{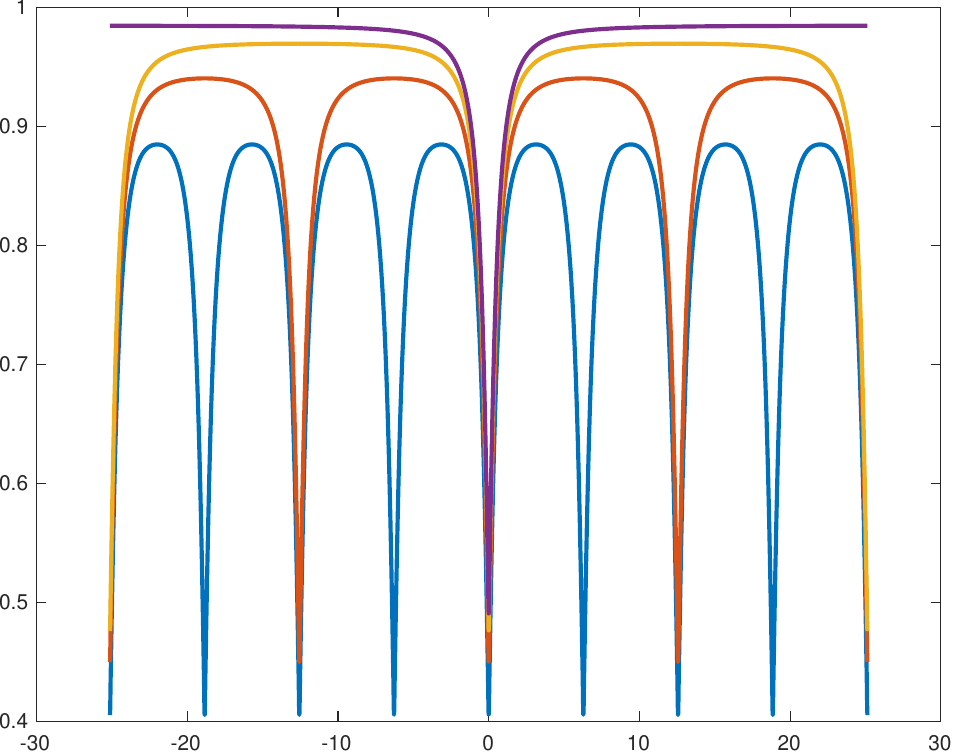}
    \caption{First $20$ terms of the cosine series of $\mu^T$. \\ 
    Blue: $T = \frac{1}{2}$. Orange: $T = \frac{1}{4}$. Yellow: $T = \frac{1}{8}$. Purple: $T = \frac{1}{16}$.}
\end{figure}

\end{example}

\begin{example}\label{ExpGrowth}

Let $H$ be the diagonal det-normalized Hamiltonian with $h_{11}(t) = \exp(t)$. The step-function approximations $h_{11}^T(t)$ satisfy \begin{equation*}
\begin{split}
    \frac{h_{11}^{T, n + 1}}{h_{11}^{T, n}} = \frac{\frac{1}{T} \int_{(n + 1)T}^{(n + 2)T} e^s ds}{\frac{1}{T} \int_{nT}^{(n + 1)T} e^s ds} = e^T,
\end{split}
\end{equation*}
which implies their corresponding spectral measures have constant Verblunsky coefficients.

Using example \ref{GeometricGrowth}, we can solve for the spectral measures $\mu_T$. For each $T$, on the interval $[0, T\pi] \subseteq \mathbb{R}$, the measure $\mu_T$ has an absolutely continuous part \begin{equation*}
    w^T(\theta) = \frac{T\left( e^T + 1 \right)}{2\left( e^T - 1 \right)} \frac{\sqrt{\frac{4*e^T}{(e^T + 1)^2} - \cos^2(T \theta)}}{\sin(T \theta)},
\end{equation*}
supported on $[\frac{1}{T} \arcsin \frac{1 - e^T}{1 + e^T}, \frac{\pi}{T} - \frac{1}{T} \arcsin \frac{1 - e^T}{1 + e^T}]$, and no singular part.

\begin{figure}[H]
    \centering
    \includegraphics[width = .45\textwidth]{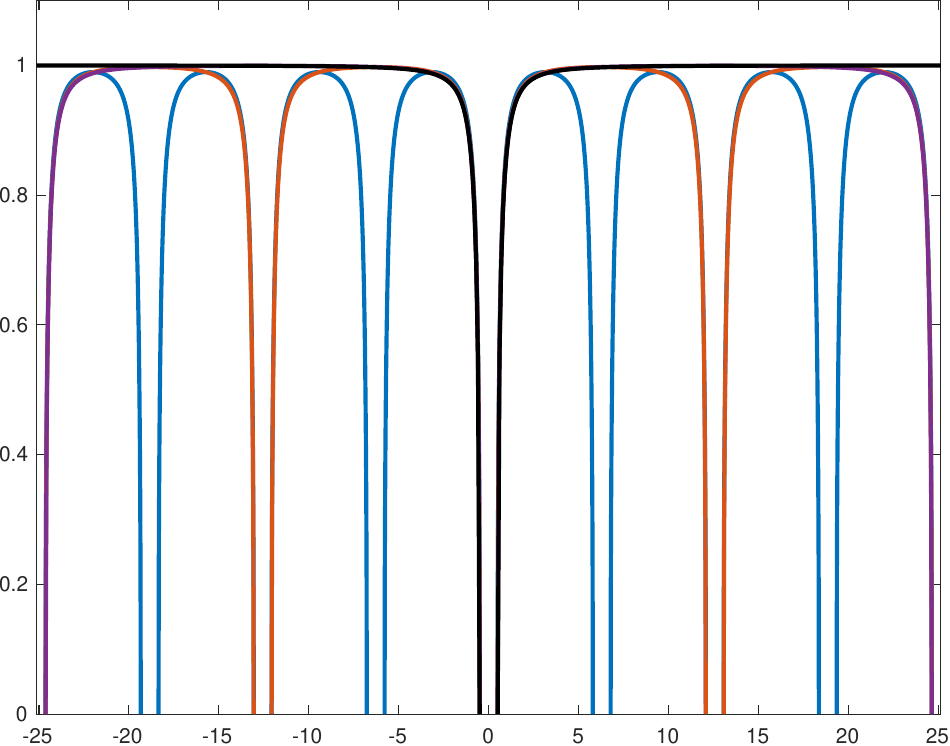}
    \caption{ $w^T(x)$ where $d^T\mu(x) = w^T(x) dx$. \\ 
    Blue: $T = \frac{1}{2}$. Orange: $T = \frac{1}{4}$. Purple: $T = \frac{1}{8}$. Black: $T = \frac{1}{16}$.}
\end{figure}

Since $w^T(x)$ is explicitly known for each $T$, we can determine the limiting measure $\mu$ as well. The limit of $\mu_T$ is given by  
\begin{equation}
    d\mu(x) = \frac{\sqrt{4x^2 - 1}}{2 \abs{x}} dx,
\end{equation}
supported on $(-\infty, -\frac{1}{2}] \cup [\frac{1}{2}, \infty)$.

\end{example}

\end{document}